\documentclass[12pt]{article}

\usepackage{amssymb}
\usepackage{amsmath}
\usepackage{amsthm}
\usepackage{diagbox}
\usepackage{enumitem}
\usepackage{subcaption}
\usepackage{multirow}

\newtheorem{theorem}{Theorem}[section]
\newtheorem{lemma}[theorem]{Lemma}
\newtheorem{proposition}[theorem]{Proposition}
\newtheorem{corollary}[theorem]{Corollary}
\newtheorem{remark}[theorem]{Remark}
\newtheorem{example}[theorem]{Example}
\newtheorem{conjecture}[theorem]{Conjecture}

\newtheorem*{constructionA}{Construction A}
\newtheorem*{constructionB}{Construction B}
\newtheorem*{constructionC}{Construction C}
\newtheorem*{thmA}{Theorem A}
\newtheorem*{thmB}{Theorem B}
\newtheorem*{thmC}{Theorem C}
\newtheorem*{thmD}{Theorem D}
\newtheorem*{thmE}{Theorem E}
\newtheorem*{analysis}{Comparative Analysis}

\begin{document}

\title{Projective planes and set multipartite Ramsey numbers for $C_4$ versus star}

\author{Claudia J. F. Gon\c{c}alves$^{a}$, \ Emerson L. Monte Carmelo$^{b}$, \\ Irene N. Nakaoka$^{c}$
}
\maketitle

\begin{abstract}
Set multipartite Ramsey numbers were introduced in 2004, ge-neralizing the celebrated Ramsey numbers.
Let $C_4$ denote the four cycle and let $K_{1,n}$ denote the star on $n+1$ vertices.
In this paper we investigate bounds on $C_4-K_{1,n}$ set multipartite Ramsey numbers.
Relationships between these numbers and the classical $C_4-K_{1,n}$ Ramsey numbers
are explored. Then several near-optimal or exact classes are derived as applications.
As the main goal, polarity graphs from projective planes allow us to find suitable subgraphs which yield some optimal classes too.
\vspace{3mm}


\vspace{3mm}


\end{abstract}

\section{Introduction}

\subsection{Ramsey numbers}
One of the most extensively studied problems in Combinatorics is that of determining the Ramsey numbers for graphs, defined as follows. Let $G_{1}$ and $G_{2}$
 be simple graphs. The \emph{graph Ramsey number}  $r(G_{1}; G_{2})$ denotes the smallest integer $c$ such that, given any
subgraph $H$ of the complete graph $K_{c}$ on $c$ vertices, the graph $H$ contains a copy of $G_{1}$ or there is a copy of $G_{2}$ in the complement of $H$.

Let $C_4$ denote the 4-cycle and let $K_{1,n}$ denote the star on $n+1$ vertices. The systematic study of the numbers $r(n):=r(C_{4}; K_{1,n})$
 was initially by Parsons \cite{Parsons} in 1975, who calculated the general upper bound $r(n)\leq n+\left\lceil \sqrt{n}\ \right\rceil +1$ and computed the exact classes
$r(n)=n+q+1$ for $n=q^{2}$ or $n=q^{2}+1$, where $q$ denotes a prime
power. Essentially, these lower bounds are derived from polarity
graphs on projective planes.

On the basis of subgraphs of such polarity graphs, Parsons
\cite{Pa0} obtained the following classes: for an odd prime power
$q$ and any $k$ such that $1\leq k\leq (q+3)/4$,
$r(q^{2}-2k)=q^{2}+q-2k+1$; if $q$ is a power of $2$, then
$r(q^{2}-k)=q^{2}+q-k+1$ for any $k$ such that $1\leq k\leq q-1$ or
$k=q+1$. The determination of exact values on $r(n)$ is a
long-standing and difficult problem. Further exact classes in
\cite{m, Wu, Z1, Z2} typically focus on the cases where $n$ is close
to $q^2$, $q(q-1)$ or $(q-1)^2$, where $q$ is a prime power.

In 2004, Burger et al. \cite{Burger} introduced the Ramsey numbers when
the host graph is a multipartite graph, more specifically,  let
$K_{c\times s}$ represent the complete multipartite graph having $c$
classes with $s$ vertices per each class. Given $s\geq 1,$ the
\emph{set multipartite Ramsey number} $M_s(G_1; G_2)$ denotes the
smallest positive integer $c$ with the property:
 for any subgraph $H$ of the multipartite graph $K_{c\times s}$, $H$ contains a copy of $G_1$ or $\overline{H}$ contains a copy of $G_2$,
  where $\overline{H}$ denotes the complement of $H$ relative to $K_{c\times s}$.
Since $K_{c \times 1}$ is isomorphic to $K_{c},$ these numbers can
be regarded as an extension of the graph Ramsey numbers, more
formally, $M_1(G_1; G_2)=r(G_1; G_2).$ The numbers $M_s(G_1; G_2)$
have been investigated when $G_1$ and $G_2$ are both multipartite
graphs \cite{Stipp, Burger}.

\subsection{Our main contributions}
In this work we focus on the numbers $M_s(n):=M_s(C_4; K_{1,n}).$
Thus $r(n)$ can be reformulated as $M_1(n).$ We describe here the
optimal classes obtained.

\begin{thmA}  \label{circulant}
Given $s\geq 2$, let $a$ be an integer such that $-1\leq a\leq \lfloor s/2\rfloor-1$. Then, for all integer $k$ in the range $1\leq k\leq s-(2a+1)$, $$M_s(ks+a)=k+2.$$
\end{thmA}

For $s$ fixed, the value of $n=ks+a$ in Theorem A belongs
to the interval $[s-1,s^2+s]$. To extend the range of the parameter
$n$, we translate the knowledge on graph Ramsey numbers into set
multipartite Ramsey numbers. For this purpose, a link between $r(n)$
and $M_s(n)$ allows us to get several near or optimal classes in
Section \ref{SecaoExatas}, for instance:

\begin{thmB} \label{ramsey1}
Let $s$ be  an odd integer, with $s\geq 3$, and suppose that $m$ is a positive integer such that $4^m\equiv 1 \, ({\rm mod}\, s)$ and $2^{m}\geq s$.
For all positive integer $k$,
$$M_s(2^{2mk}-2^{mk}-1)=\frac{2^{2mk}-1}{s}+1.$$
\end{thmB}

\begin{thmC} \label{ramsey2}
Given a prime power $q\geq 3$, we have:
\begin{enumerate}
\item [(1)] $M_2(q^2-2)=\displaystyle{\frac{q^2+q}{2}};$
\item [(2)] If $q$ is even  and $q \geq 8$, then for each odd integer $k$ in the range $3\leq k \leq q-3$,
$$M_2(q^2-k-1)=\displaystyle{\frac{q^2+q-k+1}{2}};$$
\item [(3)] If $q$ is even, then $M_2((q-1)^2-2)=\displaystyle{\frac{(q-1)^2+q-3}{2}}+1;$
\item [(4)] If $q$ is odd, then for each even integer $k$ in the range $2\leq k\leq 2\lceil q/4\rceil$,
$$M_2(q^2-k)= \displaystyle{\frac{q^2+q-k}{2}}+1.$$
\end{enumerate}
\end{thmC}

In order to study the numbers $M_s(n)$ for parameters that are not
contemplated by the results above, we propose some constructions of
graphs which can provide new exact classes. Our main approach
attempts to search multipartite subgraphs of polarity graphs from
projective planes with good impact on the numbers $M_s(n)$. For this
purpose, key ingredients for the constructions consist in: (i) an effective way to label the points and the lines of the projective plane and (ii)
an effective way to split the vertices and edges of a polarity graph into classes. Then an analysis on structural properties of the certain projective planes
allows us to obtain the main results of this work:

\begin{thmD}\label{exato_pot_primo_1}
Let $q$ be a prime power. For any $i$ in the range $0\leq i\leq q-1$,
$$M_q((q-1)(q-i)+1)=q-i+2.$$
\end{thmD}

\begin{thmE}   \label{exato_pot_primo_2}
Given a prime power $q$, we have:
\begin{enumerate}
\item [(1)] If $q\geq 5$, then $M_{q-k}((q-k)(q-2)+2)=q+1$ for any $k$ such that $0\leq k\leq\lfloor q/2\rfloor -1$;
\item [(2)] If $q\geq 3$, then $M_{q-1}((q-1)^2+1)=q+2$.
\end{enumerate}
\end{thmE}

\subsection{Organization of the work and notation}
This article is structured as follows. Inspired by a result of Chen \cite{Chen}, a relationship between $M_s(n)$ and $M_s(n+1)$ is discussed in the next section.
Section \ref{SectionUpper} presents a general upper bound on $M_s(n)$ and the proof of Theorem A. A few known exact classes on $r(n)$ are reported in Section \ref{SecaoExatas}.
By combining the previous results, several good bounds on $M_s(n)$ are obtained, including Theorems B and C.
Projective plane and its polarity are constructed on the basis of cartesian coordinates in
Section \ref{polarity}. Structural properties of polarity graphs in Section \ref{polarity_graph} allow us to obtain some classes,
including Theorems D and E. To summarize our contributions, we present a table of lower and upper bounds on $M_s(n)$
for $2\leq s\leq 5$ and $2\leq n\leq 17$ in Section \ref{tabela}.

For convenience, we collect here some of the notation on graph
theory which we use throughout the work. For a graph $G$, denote its
vertex set by $V(G)$
 and its edge set by $E(G)$. Given a subset $W$ of $V(G)$, $G[W]$ represents the subgraph of $G$ induced by $W$. As usual, the neighborhood of a vertex $v$
 is represented by $N_G(v)=\{u\in V(G) \, : \, \{u,v\}\in E(G)\}$ with degree $d_G(v)=|N_G(v)|,$ and let $N_G[v]=N_G(v)\cup\{v\}$.
 The simplified notation $H\subseteq G$ indicates that $H$ is a subgraph of $G$, while $H\nsubseteq G$ means that
$G$ does not contain a copy of $H$. We refer to the book \cite{Bollobas} for further information on graph theory.


\section{A growth property} \label{s2}

Chen \cite{Chen} proved that $r(n+1)\leq r(n)+2$ for any $n\geq 2$,
answering a question posed by Burr, Erd\H{o}s, Faudree, Rousseau,
and Schep \cite{B}. A natural question arises: what are the
relationships between $M_s(n)$ and $M_s(n+1),$ besides the trivial
$M_s(n)\leq M_s(n+1)$? We now discuss this question. An extension of
Chen's result is established below.

\begin{proposition} \label{mono}
For integers $s\geq 2$ and $n\geq 2$, $M_s(n+1)\leq M_s(n)+2$.
\end{proposition}
\begin{proof}
Let $c=M_s(n)$ and suppose for a contradiction that $M_s(n+1)>c+2$. It means that there is a subgraph $G$ of $K_{(c+2)\times s}$ such that
$G$ does not contain a copy of $C_4$ neither  $\overline{G}$ (the complement of $G$ relative to $K_{(c+2)\times s}$) contains a copy of $K_{1,n+1}$.
Write
$$V=V(K_{(c+2)\times s}) = V_1\cup \ldots \cup V_c \cup V_{c+1} \cup V_{c+2}.$$
Let us analyze the induced subgraph $H=G[V_{1} \cup \ldots \cup
V_{c}]$ of $G$ by deleting the vertices of $V_{c+1}\cup V_{c+2}$.
The choice of $c$ implies that $C_4 \subseteq H$ or
$K_{1,n}\subseteq \overline{H}$. Since $C_4\not\subseteq G$, the
graph $\overline{G}$ contains a star $K_{1,n}$. Denote by $u_1$ a
vertex in $\overline{G}$ of degree $n$. Without loss of
 generality, suppose $u_1\in V_1$. A similar reasoning shows that there is another star $K_{1,n}$ in $\overline{G}[V\setminus (V_{1}\cup V_{c+2})]$,
 whose vertex of degree $n$ is denoted by $u_2$. Such stars are contained in $\overline{G}[V\setminus (V_{c+1}\cup V_{c+2})]$
 and $\overline{G}[V\setminus (V_{1}\cup V_{c+2})]$, respectively, and $K_{1,n+1}\nsubseteq \overline{G}$, hence
\begin{equation}\label{af1}
V_{c+1}\cup V_{c+2}\subseteq N_G(u_1) \ \ \ \text{and} \ \ \ V_{1}\cup V_{c+2}\subseteq N_G(u_2).
\end{equation}
Let $x,y\in V_{c+2}$. By (\ref{af1}), the vertices $u_1xu_2y$ induce a $C_4$ in
$G$, a contradiction. Therefore, $M_s(n+1)\leq c+2$.
\end{proof}

Under certain conditions, a refinement of Proposition \ref{mono}
states that $M_s(n+1)=M_s(n)$ or $M_s(n+1)=M_s(n)+1$, more
specifically:

\begin{proposition}\label{n+1}
Given integers $s\geq 2$ and $n\geq 2,$ let $c=M_s(n)$. If $n<cs-(1+\sqrt{4s(c+1)-3})/2$, then $M_s(n+1)\leq M_s(n)+1$.
\end{proposition}
\begin{proof}
Suppose for an contradiction that $M_s(n+1)>c+1$. Then there is a subgraph $G$ of $K_{(c+1)\times s}$ such that $G$ does not contain a copy of $C_4$ and
the complement $\overline{G}$ of $G$ (relative to $K_{(c+1)\times s}$) does not contain a copy of $K_{1,n+1}$. Write
 $V=V(K_{(c+1)\times s}) = V_1\cup V_2\cup \ldots\cup V_c \cup V_{c+1}$.

The choice of $c$ assures that $C_4 \subseteq G[V\setminus V_{c+1}]$ or $K_{1,n}\subseteq \overline{G}[V\setminus V_{c+1}]$. Since
 $C_4\not\subseteq G$, the graph $\overline{G}$ contains a copy of $K_{1,n}$. Note that $\overline{G}$ does  not contain $K_{1,n+1}$. Denote by $u$ the vertex of
  degree $n$ of this star $K_{1,n}$ in $\overline{G}$ and let $N=N_G(u)$. Note that $V_{c+1}\subseteq N$.

For each $v\in N$, denote $A_v = N_G(v)\setminus\{u\}$. The fact $C_4\not\subseteq G$ implies
\begin{equation}\label{e4}
    A_v \cap A_w = \emptyset  \ \ \ \ \text{for all} \   v, w\in N \   \text{with} \  v \neq w.
\end{equation}
For any $v\in N,$  $d_G(v)+d_{\overline{G}}(v)=cs$ and $d_{\overline{G}}(v)\leq n.$ Thus $|N_G(v)| \geq cs-n$ and
\begin{equation}\label{e5}
|A_v| = |N_G(v)| - 1 \geq cs-n-1.
\end{equation}

Let $z\in V_{c+1}$. Eq. (\ref{e4}) reveals that $zx \in E(\overline{G})$ for every $x\in A_w \setminus V_{c+1}$ and for every $w\in N\setminus\{z\}$. Hence,
$$ \left( \cup_{\stackrel{w\in N}{w\neq z}} A_w  \right) \setminus V_{c+1} \subseteq N_{\overline{G}}(z).$$

Eq. (\ref{e5}) states that $|N_{\overline{G}}(z)| \geq (|N|-1)(cs-n-1)-k$ for some integer $k$, $0\leq k\leq |V_{c+1}| \leq s$. Thus
$|N_{\overline{G}}(z)| \geq (cs-n-1)^2-s$ holds. On the other hand, $|N_{\overline{G}}(z)| \leq n.$ A combination of the last two inequalities implies
$n \geq (cs-n-1)^2-s$. Elementary calculation shows us that any solution (on the variable $n$) of this quadratic inequality does not satisfies the hypothesis.
Thus $M_s(n+1)\leq c+1$.
\end{proof}

We believe that a more general result holds by removing the
hypothesis from the last proposition. The following conjecture is
proposed.

\begin{conjecture}
For any integers $s\geq 2$ and $n\geq 2$, $M_s(n+1)\leq M_s(n)+1$.
\end{conjecture}


\section{An upper bound from density argument}\label{SectionUpper}

Density arguments have been a powerful method in exploring problems
of extremal graph theory, see \cite{Bollobas} for instance. By using this method, Parsons
\cite{Parsons} obtained the general upper bound $r(n)\leq n+\lceil
\; \sqrt{n} \; \rceil +1.$ It is worth mentioning that structural
properties of graphs were also used to reach a slight refinement of
the upper bound when $n$ is a square number. An adaptation for
multipartite graphs can extend Parson's bound except for the case
where $n$ is a square number, as established below.

\begin{proposition}\label{LimSup}
For integers $s\geq 1$ and $n\geq 2$, let $t=n+s-1$. The following
upper bound holds.
$$M_s(n)\leq \left \{
\begin{array}{ll}
\displaystyle{\frac{n+\sqrt{t}}{s}+2},                         & \mbox{if}\ \ \sqrt{t}\in\mathbb{Z}\ \ \mbox{and} \ \ s\mid (n+\sqrt{t});\\
                                                               &                                                               \\
\displaystyle{\left\lceil\frac{n+\sqrt{t}}{s}\right\rceil+1},  & \mbox{otherwise}. \\
\end{array}
\right. $$
\end{proposition}
\begin{proof}
Given an integer $c$, to be estimated afterward, let $G=(V,E)$ be an
arbitrary subgraph of $K_{c\times s}$. Suppose that $\overline{G},$
the complement of $G$ relative to $K_{c\times s}$, contains no copy
of $K_{1,n}$. Hence
 $d_{\overline{G}}(v)\leq n-1$ for all $v\in V(K_{c\times s})$. Since $d_{G}(v)+d_{\overline{G}}(v)=(c-1)s$, the inequality
\begin{equation}\label{1}
    d_{G}(v)\geq (c-1)s-n+1
\end{equation}
holds for all $v\in V(K_{c\times s})$. In order to check that $G$
contains a copy of $C_4$, the density argument is essentially based on two statements below. Let $\alpha$ be
the number of copies of $K_{1,2}$ in $G$. \\

\emph{Claim 1:} If $\alpha > \binom{cs}{2}$, then there is a copy of $C_4$ in $G$.

To proceed its proof, note that a copy of $K_{1,2}$ in $G$ can be
represent by a ordered pair $(v,\{u,w\})$ such that $u\neq w$ and
$v$ is adjacent to both $u$ and $w$ in $G$. If $\alpha >
\binom{cs}{2}$, the pigeonhole principle states that at least one
pair of vertices $u\neq w$ has
more than one common neighbor, forming a copy of $C_4$ in $G$. \\

\emph{Claim 2:} If $ c> \frac{ n+\sqrt{t}}{s}+1$, then $\alpha >
\binom{cs}{2}.$

For each $v\in V$, there are $\binom{d_G(v)}{2}$ possibilities of choosing a subset with two elements of its $d_G(v)$ neighbors.
Since the binomial ${x \choose 2}$ is a convex function, Jensen's inequality and Eq. (\ref{1}) yield
\begin{equation} \label{est}
\alpha=\sum_{v\in V}\binom{d_G(v)}{2} \geq |V|\binom{\left(\sum_{v\in V}d_G(v)\right)/|V|}{2} \geq cs\binom{(c-1)s-n+1}{2}.
\end{equation}
If
$$cs \binom{(c-1)s-n+1}{2}>\binom{cs}{2},$$
then Claim 1 and Eq. (\ref{est}) assure that $G$ contains a copy of $C_4$ in $G$. A simple computing reveals that the last inequality holds provided $c>\left(n+\sqrt{n+s-1}\right)/s+1.$
This concludes the proof.
\end{proof}

A computational approach shows $M_2(4)\leq 4$, while Proposition
\ref{LimSup} produces only $M_2(4)\leq 5$. However, we will see
throughout this paper that Proposition \ref{LimSup} can be optimal
for several classes.

If the parameters $s$ and $n$ satisfy suitable arithmetic conditions, then the upper bound given by Proposition \ref{LimSup} is attained, more specifically:

\begin{proposition}\label{cond_aritmetica}
Given integers $s,\, n\geq 2$, let $t=n+s-1$. Suppose that
$(n+\sqrt{t})/s$ is not an integer and
$\bigl\lceil\left(n+\sqrt{t}\right)/s\bigr\rceil \leq \bigl\lfloor
(n+1)/s\bigr\rfloor +1$. Then
$$M_s(n)=\left\lceil \frac{n+\sqrt{t}}{s} \right\rceil+1.$$
\end{proposition}
\begin{proof}
Let $c=\bigl\lceil\left(n+\sqrt{n+s-1}\right)/s\bigr\rceil+1$. Proposition \ref{LimSup} implies the upper bound. For the lower bound, it is enough to
exhibit a subgraph $G$ of $K_{(c-1)\times s}$ such that $C_4\nsubseteq G$ and $K_{1,n}\nsubseteq \overline{G}$, where $\overline{G}$ is the
complement of $G$ relative to $K_{(c-1)\times s}$. The proof is divided into two cases:

\emph{Case 1: $c\notin \{2,3,5\}$}. Considering $G=sC_{c-1},$ the
disjoint union of $s$ cycles of length $c-1$, we obtain
 $G\subseteq K_{(c-1)\times s}$ and $C_4\nsubseteq G$. Further, as the subgraph $G$ is $2$-regular, the inequality in the hypotheses  gives us
  $d_{\overline{G}}(v)\leq n-1$ for all $v \in V(K_{(c-1)\times s})$. Consequently,  $K_{1,n}\nsubseteq \overline{G}$.

\emph{Case 2: $c\in \{2,3,5\}$}. For $c=2$, the result is obvious.
If $c=3$ and $s=2$, taking  a perfect matching $G$ in $K_{2\times
2}$, it is clear that   $C_4\nsubseteq G$ and $K_{1,n}\nsubseteq
\overline{G}$. For the remaining subcases, we take  a Hamiltonian
cycle $G=(V,E)$ defined by $V=V_1\cup V_2\cup \ldots \cup V_{c-1}$,
where $V_j=\{v_{j1},v_{j2},\ldots,v_{js}\}$, $1\leq j\leq c-1$;
$E=\bigcup_{l=1}^{s}\left\{v_{jl}v_{(j+1)l}\  :  \ 1\leq j\leq
c-2\right\}\cup \left\{v_{(c-1)k}v_{1(k+1)}\  :  \ 1\leq k\leq
s-1\right\}\cup \{v_{(c-1)s}v_{11}\}$. Thus, $G$ is a 2-regular
graph that does not contain a copy of $C_4$ and, as in the Case 1,
$d_{\overline{G}}(v)\leq n-1$ for all $v\in V(K_{(c-1)\times s})$.

The prove is complete.
\end{proof}

We are now in a position to show Theorem A.

\begin{proof}[{\bf Proof of Theorem A}]
Write $n=ks+a$ and $t=n+s-1$. The inequalities $-1\leq a\leq \lfloor s/2\rfloor-1$ imply $\lfloor (n+1)/s\rfloor +1 = k+1$. On the other hand, the hypothesis $1 \leq k\leq s-(2a+1)$ assures that $0< a + \sqrt{(k+1)s+a-1}\leq s$. These facts above produce
$$\left\lceil\frac{n+\sqrt{t}}{s}\right\rceil=k+\left\lceil\frac{a+\sqrt{(k+1)s+a-1}}{s}\right\rceil = k+1 = \left\lfloor\frac{n+1}{s}\right\rfloor + 1,$$
and the statement follows as an application of Proposition \ref{cond_aritmetica}.
\end{proof}


\section{Lower bounds from Ramsey numbers}\label{SecaoExatas}

For our purpose, we report only a few of the well-known results on $r(n)$.

\begin{theorem}\label{Exatas} Let $q$ be a prime power. The following exact classes hold:
\begin{enumerate}
    \item [(1)] $r(q^2+k)=q^2+q+1+k$, where $k\in\{0,1\}$ (\cite{Parsons});
    \item [(2)] $r(q^2-2)=q^2+q-1$ if $q\geq 3$ (\cite{Wu});
    \item [(3)] $r(q^2-q-1)=q^2$ if $q$ is even (\cite{Pa0});
    \item [(4)] $r(q^2-k-1)=q^2+q-k$ for even $q\geq 4$, where $0\leq k\leq q$ except $k\in\{1,q-1\}$ (\cite{Pa0});
    \item [(5)] $r((q-1)^2+k)=(q-1)^2+q+k$ for even $q\geq 4$, $k\in\{-2,0,1\}$ (\cite{Z2,m});
    \item [(6)] $r(q^2-k)=q^2+q-(k-1)$ for odd $q$, where $1\leq k\leq 2\lceil q/4\rceil$ and $k\neq 2\lceil q/4\rceil - 1$ (\cite{Pa0,Z1}).
    \end{enumerate}
\end{theorem}

As mentioned in Introduction, the following link between $r(n)$ and
$M_s(n)$ is very useful.

\begin{proposition}\label{lower}(\cite{Pablo})
For all integers $s\geq 1$ and $n\geq 2$,
$$\left\lfloor\frac{r(n)-1}{s}\right\rfloor +1 \leq M_s(n).$$
\end{proposition}

\begin{proof}
It is a consequence of \cite[Theorem 3]{Pablo}.
\end{proof}

Let us discuss the impact of the results above. To illustrate the first consequence, certain numbers
$M_s(n)$ can be obtained up to an error of 1.

\begin{proposition}\label{ramsey3}
Let $k\in\{0,1\}$. Given an integer $s\geq 2$ and a prime power $q\geq 4$ such that $s\leq 2(q+1)-k$, denote $L_k=(q^2+k+\sqrt{q^2+k+s-1})/s.$
If $L_k$ is not an integer,
then $\lceil L_k \rceil \leq M_s(q^2+k)\leq \lceil L_k \rceil+1.$
\end{proposition}

\begin{proof}
We prove only the case $k=0$. The case where $k=1$ follows
analogously. A combination of Proposition \ref{lower}, Theorem
\ref{Exatas} (1), and Proposition \ref{LimSup} yields
$$\left\lfloor \frac{q^2+q}{s} \right\rfloor +1 \leq M_s(q^2)\leq \bigl\lceil L_0\bigr\rceil +1.$$
Both bounds are very tight. Indeed, the hypothesis $s\leq 2(q+1)$ assures that $q^2+s-1\leq (q+1)^2$. Moreover, there exist unique $a$ and $r$ such that
 $q^2+q=as+r$, with $0\leq r<s$. These facts above produce
$$\bigl\lceil L_0 \bigr\rceil \leq \left\lceil \frac{q^2+q+1}{s} \right\rceil \leq \left\lceil \frac{(a+1)s}{s} \right\rceil = a+1
 \leq \left\lfloor \frac{q^2+q}{s} \right\rfloor +1,$$
and the result is proved. In fact, a simple argument shows that $\bigl\lceil L_0\bigr\rceil = \bigl\lfloor (q^2+q)/s\bigr\rfloor +1$.
\end{proof}

The result above reflects that either Proposition \ref{lower} or
Proposition \ref{LimSup} is optimal for those parameters.
Although the gap does not exceed $1$, Theorems B and C reveal that
both propositions are sharp for suitable classes.

\begin{proof}[{\bf Proof of Theorem B}]  Let  $s\geq 3$ be an odd integer and  $m$  a positive integer such that $4^m\equiv 1 \, ({\rm mod}\, s)$ and $2^{m}\geq s$.
Write $q=2^{mk}$, $n=q^2-q-1,$ and $t=n+s-1$. Since $s\mid (q^2-1)$, Proposition \ref{lower} and Theorem \ref{Exatas}(3) produce the lower bound
$$\frac{q^2-1}{s} + 1 \leq M_s(n).$$
It remains the upper bound. Since $3\leq s\leq q$, the inequalities $q-1 < \sqrt{t} <q$ hold, and consequently $\sqrt{t}$ is not an integer. The last
inequalities and Proposition \ref{LimSup} produce
$$M_s(n)\leq \frac{q^2-1}{s} + 1 + \left\lceil \frac{-q+\sqrt{t}}{s} \right\rceil = \frac{q^2-1}{s} + 1,$$
completing the proof.
\end{proof}

Given an integer $s\geq 1$, recall that the Euler phi function $\phi(s)$ is defined as the number of integers $a$ in
the range $1\leq a\leq s$ in which $gcd(a,s)=1$, where $gcd(a,s)$ denotes the greatest common divisor of $a$ and $s$.

\begin{corollary}\label{ramsey1b}
Let $s\geq 3$ be an odd integer. For any positive integer $k$,
$$M_s(2^{2\phi(s)k}-2^{\phi(s)k}-1)=\frac{2^{2\phi(s)k}-1}{s}+1.$$
\end{corollary}
\begin{proof}
Since $gcd(s,2)=1$, Euler's Theorem on congruence states that $2^{\phi(s)}\equiv 1 \, ({\rm mod} \, s)$, and consequently, $4^{\phi(s)}\equiv 1 \, ({\rm mod} \, s)$
 and $2^{\phi(s)}\geq s.$ The proof follows as an immediate application of Theorem B.
\end{proof}

\begin{proof}[{\bf Proof of Theorem C}]
First let us prove item (2). Let $q\geq 4$ be a power of $2$ and $k$  an odd integer such that $3\leq k\leq q-3$. Put $s=2$ and $n=q^2-k-1$. The inequalities
$q-1<\sqrt{q^2-k}<q$ reveals that $\sqrt{q^2-k}$ is not an integer.
A combination of Proposition \ref{lower}, Theorem \ref{Exatas}(4),
Proposition \ref{LimSup} and the last inequality yields
$$ \left\lfloor \frac{n+q}{2} \right\rfloor +1  \leq M_2(n) \leq \left\lceil \frac{n+\sqrt{q^2-k}}{2} \right\rceil+1 \leq
\left\lceil \frac{n+q}{2} \right\rceil+1.$$
The choice of $k$ assures that $n+q$ is even and, consequently, $2\mid (n+q)$, proving the statement.

A similar reasoning shows items (1), (3) and (4); their proofs use,
respectively,  parts (2), (5) and (6) of Theorem \ref{Exatas}.
\end{proof}


\section{Projective plane }\label{polarity}

Typically, the works \cite{Brown, Erdos, Parsons, Pa0, Wu, Z1} make
use of ``homogeneous coordinates'' to describe projective planes as
well as polarity graphs. Instead of that approach, we explore
constructions based on ``cartesian coordinates''. The reader is
referred to \cite{HP} for definitions not included in this paper and for
further information on projective planes and their polarities.

Given a prime power $q$, let $\mathbb{F}_q$ denote a finite field
with $q$ elements. Consider the geometric structure
$\Pi_{q}=(X,\mathcal{B})$ formed by the set of \emph{points}
$$ X=(\mathbb{F}_q\times\mathbb{F}_q)\cup (\{q\}\times \mathbb{F}_q)\cup \{(q,q)\}$$
and the set of \emph{lines} $\mathcal{B}$ which are indexed by the points according to the rule:
for each $(x,y)\in \mathbb{F}_q\times\mathbb{F}_q$ and for each $z\in\mathbb{F}_q$, define

\begin{equation} \label{linhas}
\begin{array}{l}
   B_{(x,y)}=\{(k,kx-y) \,:\,  k\in\mathbb{F}_q\}\cup\{(q,x)\}, \\
   B_{(q,z)}=\{z\}\times \mathbb{F}_q\cup \{(q,q)\} \\
   B_{(q,q)}=\{q\}\times \mathbb{F}_q\cup\{(q,q)\}. \\
\end{array}
\end{equation}

\begin{lemma} \label{plano} The structure $\Pi_{q}=(X,\mathcal{B})$ is a projective plane of order $q$.
\end{lemma}
\begin{proof} We have to verify that $\Pi_{q}$ satisfies the three axioms below.

\emph{Axiom 1:} Any two distinct points lie on precisely one line.
Its proof is divided into four situations: \emph{Case 1}: if both points $P=(x_1,y_1)$ and $Q=(x_2,y_2)$ belong to $\mathbb{F}_q\times\mathbb{F}_q$.
When $x_1=x_2$, the only line containing them is $B_{(q,x_1)}$. Suppose $x_1\neq x_2$. Note that $(x_1,y_1)\in B_{(a,b)}$ if and only if $y_1=x_1a-b$,
and $(x_2,y_2)\in B_{(a,b)}$ if and only if $y_2=x_2a-b$. The unique solution $a=(y_1-y_2)/(x_1-x_2)$ and $b=(y_1-y_2)/(x_1-x_2)x_1-y_1$ reveals that
$B_{(a,b)}$ is the only line containing both $P$ and $Q$.
\emph{Case 2}: if $P=(x_1,y_1)\in\mathbb{F}_q\times\mathbb{F}_q$ and $Q=(q,z)\in\{q\}\times\mathbb{F}_q$. In this case, $B_{(z,x_1z-y_1)}$ is the  only line that
 contains $P$ and $Q$.
\emph{Case 3}: if $P=(x_1,y_1)\in\mathbb{F}_q\times\mathbb{F}_q$ and $Q=(q,q)$, then the only line containing them is $B_{(q,x_1)}$.
\emph{Case 4}: Finally, $(q,z)\in\{q\}\times\mathbb{F}_q$ and $(q,q)$ lie on the line $B_{(q,q)}.$

\emph{Axiom 2:}  Any two distinct lines meet in a unique point.
The proof is similar to that presented in Axiom 1.

\emph{Axiom 3:} There exists a quadrilateral.
Indeed, consider the set of points $C=\{(0,0),(1,1),(q,0),(q,q)\}$.
By simple inspection, any three points in $C$ are not incident with
a common line.

Therefore $\Pi_{q}$ is a projective plane of order $q$.
\end{proof}

\begin{remark}\label{remark1}
In view of Eq. (\ref{linhas}), note that $B_{(q,q)}$ corresponds to the \emph{line at infinity} composed by the $q+1$ \emph{points at infinity}, namely, $(q,y)$
 for any $y \in \mathbb{F}_q$ and $(q,q)$.
A classical result (see \cite{HP}, for instance) states that an affine plane $(\tilde{X},\mathcal{\tilde{B}})$ is induced by deleting the line $B_{(q,q)}$
 and its points. More precisely, let $\tilde{X}=X\setminus B_{(q,q)}$ and
 $\tilde{\mathcal{B}}=\{B_P\setminus B_{(q,q)}\,:\, B_P\in \mathcal{B},\ \ P\neq (q,q)\}$. For each $x\in\mathbb{F}_q$, the set of $q$ new lines
  $\{B_{(x,y)}\setminus\{(q,x)\}\,: \, y\in\mathbb{F}_q\}$ constitutes a \emph{parallel class} (a partition) of $\tilde{X}.$ The set
$\{B_{(q,z)}\setminus\{(q,q)\} \,: \, z\in\mathbb{F}_q\}$ is also a parallel class.
\end{remark}

Recall that a \emph{polarity} $\sigma$ of a projective plane $\Pi=(X,\mathcal{B})$ is a bijection
 $\sigma: X\cup \mathcal{B} \rightarrow X\cup \mathcal{B}$ which satisfies the properties: \emph{(i)} $\sigma(X)=\mathcal{B}$ (thus $\sigma(\mathcal{B})=X$),
\emph{(ii)} $\sigma^2$ is the identity ($\sigma$ is an involution), and \emph{(iii)} for any point $P$ and any line $B$, $P\in B$ if and only if
$\sigma(B) \in \sigma(P).$ Moreover, a point $P$ (resp., a line $B$) is \textit{absolute} (with respect to $\sigma$)
if $P\in\sigma(P)$ (resp., $\sigma(B)\in B$).

 \begin{lemma} \label{polaridade}
The projective plane $\Pi_{q}$ in Lemma \ref{plano} admits an orthogonal polarity. If $q$ is an odd prime power, the set of the absolute points
 $\{(x,y)\in\mathbb{F}_q\times\mathbb{F}_q \,:\, x^2=2y\}\cup\{(q,q)\}$ generates an oval. If $q$ is even, the set of the absolute points is
 the line $\{(0,y)\,:\, y\in\mathbb{F}_q\}\cup\{(q,q)\}.$
\end{lemma}

\begin{proof}
The key ingredient for the construction of $\sigma$ is inspired by Eq. (\ref{linhas}). Consider the bijection
$\sigma: X\cup \mathcal{B} \rightarrow X\cup \mathcal{B}$ defined by $\sigma(P)=B_P$ and $\sigma(B_P)=P$ for all $P\in X$.
The bijection satisfies naturally the properties (i) and (ii). In order to proceed the proof of condition (iii), we analyze five cases.
Case 1: if $P\in\mathbb{F}_q\times\mathbb{F}_q$ and $Q=(q,q)$, then $P\notin B_Q$ and $Q\notin B_P$ and the equivalence holds. Case 2: the case where
both $P$ and $Q$ belong to $\{q\}\times \mathbb{F}_q$ is analogous to the previous one. Case 3: the case where $P\in\{q\}\times \mathbb{F}_q$
and $Q=(q,q)$ is obtained directly by the definitions of $B_P$ and $B_{(q,q)}$. Case 4: for $P=(x,y)$ and $Q=(a,b)$ in $\mathbb{F}_q\times\mathbb{F}_q$,
note that $(x,y)\in B_{(a,b)}$ is equivalent to $(a,b)\in B_{(x,y)}$, by the construction of the lines in Eq. (\ref{linhas}). Case 5: it remains
 the case $P=(x,y)\in\mathbb{F}_q\times\mathbb{F}_q$ and $Q=(q,z)$ where $z\in\mathbb{F}_q$. A look on Eq. (\ref{linhas}) reveals that $P\in B_Q$ if and only if
 $P=(z,y)$, which is equivalent to $Q\in B_P$. Thus $\sigma$ is a polarity of $\Pi_{q}$.

Note that $(q,q)$ is always an absolute point but a point in $\{q\}\times\mathbb{F}_q$ is not absolute.
For a point $(x,y)\in \mathbb{F}_q\times \mathbb{F}_q$, note that $(x,y)\in B_{(x,y)}$ if and only the point $(x,y)$ belongs to the curve $x^2=2y$.
The conclusion is immediate when $q$ is odd. If $q$ is even,
$\mathbb{F}_q$ is a field of characteristic 2, hence $2y=0$ for all
$y\in\mathbb{F}_q$. Thus $x^2=2y$ if and only if $x=0$, that is, $(0,y)$ is an absolute point for any $y\in\mathbb{F}_q$.
The proof is complete.
\end{proof}

\begin{example}\label{Pi_3}
\emph{We illustrate the incidence matrix of the projective plane
$\Pi_3$. Select the primitive element $\alpha=\overline{2}$ from the
field $\mathbb{Z}_3.$ The rows and columns of the matrix are labeled
by using the same sequence
$$(0,0),(0,\alpha),(0,\alpha^2),\ldots,(3,3),$$ according to the
table below. This choice of  the order will be explained in the next
section. The symbol $\underline{1}$ indicates an absolute point of
$\Pi_3$.}

\emph{
\begin{footnotesize}
\begin{center}
\begin{tabular}{|c||ccc|ccc|ccc|ccc|c|} \hline
(0,0) & $\underline{1}$ & 0 & 0 & 1 & 0 & 0 & 1 & 0 & 0 & 1 & 0 & 0 & 0 \\
(0,$\alpha$) & 0 & 0 & 1 & 0 & 0 & 1 & 0 & 0 & 1 & 1 & 0 & 0 & 0 \\
(0,$\alpha^2$) & 0 & 1 & 0 & 0 & 1 & 0 & 0 & 1 & 0 & 1 & 0 & 0 & 0 \\ \hline
($\alpha$,0) & 1 & 0 & 0 & 0 & 0 & 1 & 0 & 1 & 0 & 0 & 1 & 0 & 0 \\
($\alpha$,$\alpha$) & 0 & 0 & 1 & 0 & $\underline{1}$ & 0 & 1 & 0 & 0 & 0 & 1 & 0 & 0 \\
($\alpha$,$\alpha^2$) & 0 & 1 & 0 & 1 & 0 & 0 & 0 & 0 & 1 & 0 & 1 & 0 & 0 \\ \hline
($\alpha^2$,0) & 1 & 0 & 0 & 0 & 1 & 0 & 0 & 0 & 1 & 0 & 0 & 1 & 0 \\
($\alpha^2$,$\alpha$) & 0 & 0 & 1 & 1 & 0 & 0 & 0 & $\underline{1}$ & 0 & 0 & 0 & 1 & 0 \\
($\alpha^2$,$\alpha^2$) & 0 & 1 & 0 & 0 & 0 & 1 & 1 & 0 & 0 & 0 & 0 & 1 & 0 \\ \hline
(3,0) & 1 & 1 & 1 & 0 & 0 & 0 & 0 & 0 & 0 & 0 & 0 & 0 & 1 \\
(3,$\alpha$) & 0 & 0 & 0 & 1 & 1 & 1 & 0 & 0 & 0 & 0 & 0 & 0 & 1 \\
(3,$\alpha^2$) & 0 & 0 & 0 & 0 & 0 & 0 & 1 & 1 & 1 & 0 & 0 & 0 & 1 \\ \hline
(3,3) & 0 & 0 & 0 & 0 & 0 & 0 & 0 & 0 & 0 & 1 & 1 & 1 & $\underline{1}$ \\ \hline
\end{tabular}
\end{center}
\end{footnotesize}}
\end{example}


\section{Polarity graph on projective plane}\label{polarity_graph}

The main interest in polarity graphs comes from the seminal papers
by Brown \cite{Brown} and independently by Erd\"{o}s, R\'{e}nyi, and
S\'{o}s \cite{Erdos}, who presented ingenious connections between
projective geometries and certain Tur\'{a}n numbers.

\subsection{The first construction and structural properties}

Given  a prime power $q$, from now on consider $\Pi=\Pi_{q}$ as the projective plane in Lemma \ref{plano} and $\sigma$ as the polarity in Lemma
\ref{polaridade}. The \textit{polarity graph} $G_{q}$ can be induced by the pair $(\Pi,\sigma)$ as follows: the vertex set is $X$ and the edge set is
$$\{\{P,Q\}\,:\, P\neq Q\ \ \mbox{and}\ \ P\in\sigma(Q)\}.$$

This graph has $q^2+q+1$ vertices, where $q+1$ vertices have degree $q$ (corresponding to the $q+1$ absolute points, by Lemma \ref{polaridade})
 and the others $q^2$ vertices with degree $q+1$. A classical theorem by Baer states that a polarity of a projective plane of order $q$ has at least $q
 +1$ absolute points, see \cite{HP}. Hence there is no hope in finding a polarity graph from a projective plane with more edges than $G_q$.

 The graph $G_q$ does not contain a copy of $C_4$. Indeed, suppose for a contradiction that
there are distinct points  $P,Q,R,S\in X$ such that
$\{P,Q\}\subseteq B_R\cap B_S$. Thus the absurd $R=S$ holds,
remembering that any pair of points lie on one common line.

\begin{example}\label{G_3}
\emph{The adjacency matrix of $G_3$ is obtained from the matrix of Example \ref{Pi_3} replacing the symbols $\underline{1}$ by $0$.}
\end{example}

We intent to find a sequence of suitable induced subgraphs of $G_q$. For this purpose, an appropriate labeling of vertices
combined with a partition of the vertices plays a central hole. Write $\mathbb{F}_q=\{0,\alpha,\alpha^2,\ldots,\alpha^{q-1}\}$, where $\alpha$ is a primitive element
of $\mathbb{F}_q$.

\begin{constructionA}\label{constructionA}
\emph{Given a prime power $q$, consider the following induced subgraph of $G_q.$
$$G(q,0,0):=G_q[X\setminus \{(q,q)\}].$$
Moreover, let $V_0=\{0\}\times \mathbb{F}_q$, $V_j=\{\alpha^j\}\times \mathbb{F}_q$ for
every $j\in\{1,2,\ldots,q-1\}$, and $V_q=\{q\}\times \mathbb{F}_q$. Thus $V_0, V_1, \ldots, V_{q}$ induce a partition of
$V:=X\setminus \{(q,q)\}.$ Two vertices $P$ and $Q$ are adjacent
 in $G(q,0,0)$ if $P\in B_Q$, $P\in V_j$ and $Q\in V_l$ for distinct $j,l\in\{0,1,\ldots,q\}$.}
\end{constructionA}

\begin{lemma}\label{lemma2}
The graph $G(q,0,0)$ on $q^2+q$ vertices satisfies the following structural properties:
\begin{enumerate}
\item [(1)] $G(q,0,0)\subseteq K_{(q+1)\times q}$ and $C_4\nsubseteq G(q,0,0)$;
\item [(2)] For distinct $P$ and $Q$ of $V$, $|N_{G(q,0,0)}(P)\cap N_{G(q,0,0)}(Q)|\leq 1$;
\item [(3)] Let $P\in V_j$ with $0\leq j\leq q-1$. For each $l\in\{0,1,\ldots, q\}$ with $l\neq j$, $P$ is adjacent to exactly one vertex of $V_l$;
\item [(4)] Let $P=(q,y)\in V_{q}$. If $y=\alpha^j$ for some $j\in\{1,2,\ldots,q-1\}$ (resp., $y=0$),
then the $q$ vertices adjacent to $P$ are those in $V_j$ (resp., $V_0$).
\item [(5)] Let $\overline{G(q,0,0)}$ be the complement of $G(q,0,0)$ relative to $K_{(q+1)\times q}$. For all $P\in V$,
 $d_{G(q,0,0)}(P)=q$ and $d_{\overline{G(q,0,0)}}(P)=q^2-q.$
\end{enumerate}
\end{lemma}

\begin{proof}
For item (1), Remark \ref{remark1} and definition of $G_q$ imply $G(q,0,0)\subseteq K_{(q+1)\times q}$.
Since $G(q,0,0)$ is a subgraph of $G_q$ and $C_4\nsubseteq G_q$, we conclude that $C_4\nsubseteq G(q,0,0)$.
Item (2) is an immediate consequence of item (1). For item (3), write $P=(a,b)$. The analysis depends on $l$. Given $l\in\{0,1,\ldots, q-1\}$, $l\neq j$,
select the point $Q_l=(x_l,x_la-b)$ (with $x_l=0$ if $l=0$, and $x_l=\alpha^l$ if $l\neq 0$). If $l=q$, select then $Q_q=(q,a)$.
Remark \ref{remark1} yields the uniqueness of the vertex $Q_l$. Item (4) follows from a look on the definitions of $V_j$ and $B_{(q,\alpha^j)}$
(resp., $V_0$ and $B_{(q,0)}$) and from Remark \ref{remark1}. Item (5) is derived from items (3) and (4) and
the equation $d_{\overline{G(q,0,0)}}(P) + d_{G(q,0,0)}(P)= q^2$.
\end{proof}

\begin{example}\label{G_3_0}
\emph{We return to Example \ref{G_3}. By Construction A applied to $G_3$, the adjacency matrix of the graph $G(3,0,0)$ is displayed
in order to illustrate its properties.  Proposition \ref{lower} and Theorem \ref{Exatas}(2) applied when $q=3$ yields only $M_3(7)\geq 4$, while the
graph $G(3,0,0)$ provides us the optimal lower bound $M_3(7)\geq 5$. A more general result is established in the next result.}
\emph{
\begin{footnotesize}
\begin{center}
\begin{tabular}{|c||ccc|ccc|ccc|ccc|} \hline
(0,0) & 0 & 0 & 0 & 1 & 0 & 0 & 1 & 0 & 0 & 1 & 0 & 0  \\
(0,$\alpha$) & 0 & 0 & 0 & 0 & 0 & 1 & 0 & 0 & 1 & 1 & 0 & 0  \\
(0,$\alpha^2$) & 0 & 0 & 0 & 0 & 1 & 0 & 0 & 1 & 0 & 1 & 0 & 0  \\ \hline
($\alpha$,0) & 1 & 0 & 0 & 0 & 0 & 0 & 0 & 1 & 0 & 0 & 1 & 0  \\
($\alpha$,$\alpha$) & 0 & 0 & 1 & 0 & 0 & 0 & 1 & 0 & 0 & 0 & 1 & 0  \\
($\alpha$,$\alpha^2$) & 0 & 1 & 0 & 0 & 0 & 0 & 0 & 0 & 1 & 0 & 1 & 0  \\ \hline
($\alpha^2$,0) & 1 & 0 & 0 & 0 & 1 & 0 & 0 & 0 & 0 & 0 & 0 & 1  \\
($\alpha^2$,$\alpha$) & 0 & 0 & 1 & 1 & 0 & 0 & 0 & 0 & 0 & 0 & 0 & 1 \\
($\alpha^2$,$\alpha^2$) & 0 & 1 & 0 & 0 & 0 & 1 & 0 & 0 & 0 & 0 & 0 & 1  \\ \hline
(3,0) & 1 & 1 & 1 & 0 & 0 & 0 & 0 & 0 & 0 & 0 & 0 & 0  \\
(3,$\alpha$) & 0 & 0 & 0 & 1 & 1 & 1 & 0 & 0 & 0 & 0 & 0 & 0  \\
(3,$\alpha^2$) & 0 & 0 & 0 & 0 & 0 & 0 & 1 & 1 & 1 & 0 & 0 & 0  \\ \hline
\end{tabular}
\end{center}
\end{footnotesize}}
\end{example}

\begin{proposition}\label{exato_pot_primo_0}
For a prime power $q$, $M_q(q^2-q+1)=q+2$.
\end{proposition}

\begin{proof}
Let $s=q$, $n=q^2-q+1,$ and $t=n+s-1$. Since $n+\sqrt{t}=q^2+1$ and $q\nmid (q^2+1)$, Proposition \ref{LimSup} produces the upper bound
$M_q(n)\leq q+2.$ Now, items (1) and (5) of Lemma \ref{lemma2} ensure that $G(q,0,0)$ is a subgraph of $K_{(q+1)\times q}$
containing no $C_4$ such that $K_{1,q^2-q+1}\nsubseteq \overline{G(q,0,0)}$. Thus $M_q(n)>q+1$, completing the proof.
\end{proof}

\subsection{The second construction and the proof of Theorem D}

The graph $G(q,i,0)$ defined below is obtained by removing the last $i$ vertex classes from the vertex set of $G(q,0,0)$, more formally:

\begin{constructionB}\label{constructionB}
\emph{Let $q$ be a prime power. For each integer $i$, $0\leq i\leq q-1$, define the induced subgraph of $G(q,0,0)$
$$G(q,i,0):=G(q,0,0)\left[V_0\cup V_1\cup\ldots\cup V_{q-i}\right].$$}
\end{constructionB}

We are ready to prove Theorem D, which is a generalization of Proposition \ref{exato_pot_primo_0}.

\begin{proof}[{\bf Proof of Theorem D}]  Let $q$ be a  prime power and $i\in \{0,1,\ldots, q-1\}$. First, let us show that $M_q((q-1)(q-i)+1)\leq q-i+2.$
For $q\in \{2,3\}$, the upper bound follows straightforward from Proposition \ref{LimSup}. Suppose $q\geq 4$. Let us study the behavior of $t=n+s-1$, where $s=q$ and $n=(q-1)(q-i)+1$. 
There is an integer $a$, $0\leq a\leq \lfloor i/2\rfloor$, such that
\begin{equation}\label{desig}
(q-i+a)^2\leq t\leq (q-i+a+1)^2.
\end{equation}
Firstly, suppose that $\sqrt{t}$ is an integer. Elementary properties of congruences imply that $n+\sqrt{t} \equiv a+1 \, ({\rm mod}\,  q)$ or
 $n+\sqrt{t} \equiv a+2 \, ({\rm mod} \,  q)$. Note that $q\nmid (n+\sqrt{t})$ since $a\leq \lfloor i/2\rfloor\leq \lfloor (q-1)/2\rfloor $.
Hence Proposition \ref{LimSup} and Eq. (\ref{desig}) yield
$$M_q(n)\leq \left\lceil \frac{n+\sqrt{t}}{q} \right\rceil +1 = q-i+ \left\lceil \frac{i+1+\sqrt{t}}{q} \right\rceil = q-i+2.$$
A similar argument shows the upper bound when $\sqrt{t}$ is not an integer.

For the lower bound, let $H=G(q,i,0)$ defined in Construction B. By Lemma \ref{lemma2}(1) and (5), $H$ is a subgraph of
$K_{(q-i+1)\times q}$ that does not contain $C_4$ and $d_{H}(P)=q-i$ for all vertex $P$ in $H$. Denote the complement of $H$ relative
to $K_{(q-i+1)\times q}$ by $\overline{H}$. An application of the equation $d_{\overline{H}}(P) + d_{H}(P)= (q-i)q$ gives us $d_{\overline{H}}(P)=(q-1)(q-i)$.
It means that $K_{1,(q-1)(q-i)+1}\nsubseteq\overline{H}$. Hence $M_q(n)>q-i+1,$ completing the proof.
\end{proof}

We observe  that the upper bound given by
Proposition \ref{LimSup} is also achieved in the class of parameters considered
 in Theorem D.

\begin{example}
\emph{Removing the last three rows and columns from the matrix of Example \ref{G_3_0}, we obtain the adjacency matrix of $G(3,1,0)$.
Moreover, removing the last six rows and columns, the adjacency matrix of $G(3,2,0)$ is obtained. These graphs provide us the optimal
lower bounds $M_3(5)>3$ and $M_3(3)>2$.}
\end{example}

\begin{analysis}  \emph{ Theorem D might improve the lower bound obtained from Proposition \ref{lower} and Theorem \ref{Exatas}. For instance, let $n=(q-1)(q-i)+1$ and suppose that
$n=p^{2a}+1$ for some prime $p.$ Thus $q-1=p^k$ and $q-i=p^l$, where $k+l=2a$ and $l\leq k$ (since $1\leq i\leq q-1$). An application of
Proposition \ref{lower} and Theorem \ref{Exatas}(1) yields
$$
M_q(n) \geq \displaystyle{\left\lfloor \frac{r(n)-1}{q} \right\rfloor + 1} =
\displaystyle{ p^l+1+\left\lfloor \frac{p^{k/2+l/2}-p^l+1}{p^k+1} \right\rfloor} =
 p^l+1= q-i+1.
$$
In contrast, Theorem D give us $M_q(n)= q-i+2.$  Further, it is worth mentioning that for  $i$ in the range $0\leq i \leq q-3$, the parameters $(s, n)=(q,(q-1)(q-i)+1)$ do not satisfy the hypotheses of Proposition \ref{cond_aritmetica}. Thus, Theorem D establishes a new class of exact values on $M_s(n)$ which can not derived from Proposition \ref{cond_aritmetica} or Theorem A.}
\end{analysis}

\subsection{The third construction and the proof of Theorem E}

We investigate only regular subgraphs until now, but non-regular
subgraphs can bring us good lower bounds too.

\begin{constructionC}\label{constructionC}
\emph{Let $q\geq 3$ be a prime power. We set
$$G(q,0,1):=G(q,0,0) \left[ \left(V_0\cup V_1\cup\ldots\cup V_{q}\right)\setminus N_{G(q,0,0)}\left[ (0,\alpha^{q-1}) \right] \right].$$
More generally, for integers $i$ and $k$ such that $1\leq i\leq q-1$
and $1\leq k\leq q-2$, define 
$$G(q,i,k):=G(q,0,0)\left[\displaystyle{\bigcup_{j=0}^{q-i} V_j}\ \ \setminus\ \ \displaystyle{\bigcup_{l=1}^{k} N_{G(q,0,0)}\left[(0,\alpha^{q-l})\right]} \right].$$}
\end{constructionC}

\begin{remark} \label{contagem}
Items (1) and (3) of Lemma \ref{lemma2} ensure that $G(q,0,1)$ is  a subgraph of  $K_{(q+1)\times(q-1)}$. Also, by definition of $G(q,0,0)$, we have  $(q,0)\in N_{G(q,0,0)}(0,y)$, for any $y\in \mathbb{F}_q$. Thus, if $1\leq i \leq q-1$ and $1 \leq k \leq q-2$,  items (2) and (3) of Lemma \ref{lemma2} yield that  $G(q,i,k)$  is  a subgraph of  $K_{(q-i+1)\times(q-k)}$.
\end{remark}

\begin{example}
\emph{The adjacency matrix of $G(3,0,1)$ is presented below. By Construction C and Example \ref{G_3_0},
$$V(G(3,0,1))=\{(0,0),(0,\alpha),(\alpha,0),(\alpha,\alpha^2),(\alpha^2,0),(\alpha^2,\alpha^2),(3,\alpha),(3\alpha^2)\};$$
that is, $V(G(3,0,1))=V(G(3,0,0))\setminus N_{G(3,0,0)}[(0,\alpha^2)]$. The graph $G(3,0,1)$ induces the optimal lower bound $M_2(5)>4$ (note that $M_2(5)\leq 5$,
 by Proposition \ref{LimSup}). Removing the last two rows and columns of
 this matrix, the adjacency matrix of $G(3,1,1)$ appears. Removing the last four rows and columns we obtain
 the adjacency matrix of $G(3,2,1)$. The graphs $G(3,1,1)$ and $G(3,2,1)$ yield $M_2(4)>3$ and $M_2(3)>2$, respectively.}
\emph{
\begin{footnotesize}
\begin{center}
\begin{tabular}{|c||cc|cc|cc|cc|} \hline
(0,0) & 0 & 0 & 1 & 0 & 1 & 0 & 0 & 0  \\
(0,$\alpha$) & 0 & 0 & 0 & 1 & 0 & 1 & 0 & 0  \\ \hline
($\alpha$,0) & 1 & 0 & 0 & 0 & 0 & 0 & 1 & 0  \\
($\alpha$,$\alpha^2$) & 0 & 1 & 0 & 0 & 0 & 1 & 1 & 0  \\ \hline
($\alpha^2$,0) & 1 & 0 & 0 & 0 & 0 & 0 & 0 & 1  \\
($\alpha^2$,$\alpha^2$) & 0 & 1 & 0 & 1 & 0 & 0 & 0 & 1  \\ \hline
(3,$\alpha$) & 0 & 0 & 1 & 1 & 0 & 0 & 0 & 0  \\
(3,$\alpha^2$) & 0 & 0 & 0 & 0 & 1 & 1 & 0 & 0  \\ \hline
\end{tabular}
\end{center}
\end{footnotesize}}
\end{example}

We restate Theorem E for reader's convenience:

 \vspace*{3mm}

\noindent{\bf Theorem E.} {\emph{
Let $q$ be a prime power. The following bounds hold:
\begin{enumerate}
\item [(1)] If $q\geq 5$, then $M_{q-k}((q-k)(q-2)+2)=q+1$ for any $k$ such that $0\leq k\leq\lfloor q/2\rfloor -1$;
\item [(2)] If $q\geq 3$, then $M_{q-1}((q-1)^2+1)=q+2$.
\end{enumerate}}
\begin{proof}
Item (1): Firstly we prove the upper bound. Consider $q\geq 5$ and $0\leq k\leq\lfloor q/2\rfloor -1$. In order to apply Proposition \ref{LimSup},
 let $s=q-k$ and $n=(q-k)(q-2)+2$. Thus $t=n+s-1=(q-k)(q-1)+1$ and we have
\begin{equation}\label{desig2}
((q-k)-1))^2\leq t\leq ((q-k)+ (k+1)/2)^2.
\end{equation}
We claim that $(k+1)/2<q-k-2$. Suppose for a contradiction that
$(k+1)/2\geq q-k-2$. This inequality implies  that $k\geq (2q-5)/3$
and combining with the hypothesis $k\leq q/2-1$, we derive the
absurd $q\leq 4$.

Eq. (\ref{desig2}) and $(k+1)/2<q-k-2$ assure that $q-k < 2+\sqrt{t} < 2(q-k)$. This fact and Proposition \ref{LimSup} produce
$$M_{q-k}(n)\leq q-1+\left\lceil \frac{2+\sqrt{t}}{q-k} \right\rceil = q+1.$$

For the lower bound, let $H=G(q,1,k)$ be the graph defined in Construction C.  Remark \ref{contagem} and Lemma \ref{lemma2}(1) state that
$H\subseteq K_{q\times(q-k)}$ and $C_4\nsubseteq H$.

We claim that $d_H(P)\geq q-1-k$ for any vertex $P$ in $H$. Indeed, $d_{G(q,1,0)}(P)=q-1$ was showed in the proof of
Theorem D. For each $l$ such that $1\leq l\leq k$, Lemma \ref{lemma2}(2) ensures that there exists at most one vertex
$Q\in  N_{G(q,0,0)}\left[(0,\alpha^{q-l})\right]$ such that $Q\in N_{G(q,0,0)}(P)$. Since the  $k$ vertex sets
$N_{G(q,0,0)}\left[(0,\alpha^{q-l})\right]$, $1\leq l\leq k$, were removed from $G(q,1,0)$ in order to construct $H$, we conclude that
at most $k$ neighbors of $P$ were removed. Thus $d_H(P)\geq q-1-k$.

Denote the complement of $H$ relative to $K_{q\times(q-k)}$ by $\overline{H}$. The last inequality and the equation
$d_{\overline{H}}(P)+d_H(P)=(q-1)(q-k)$ imply
$d_{\overline{H}}(P)\leq (q-k)(q-2)+1$ for any vertex $P$ in $H$, thus $K_{1,n}\nsubseteq \overline{H}$. Therefore, $M_{q-k}(n)>q.$

Item (2): In this case, let $s=q-1$ and $n=(q-1)^2+1$. Hence $t=(q-1)^2+q-1$ is bounded by $(q-1)^2< t < q^2$. An application of
Proposition \ref{LimSup} produces
$$M_{q-1}(n)\leq \left\lceil \frac{n+\sqrt{t}}{q-1} \right\rceil +1 = q + \left\lceil \frac{1+\sqrt{t}}{q-1} \right\rceil \leq q+2.$$
For the lower bound, consider the graph $J=G(q,0,1)$ given in Construction C. Remark \ref{contagem} and Lemma \ref{lemma2}(1)  state that $J$ is a subgraph of
 $K_{(q+1)\times(q-1)}$ and $C_4\nsubseteq J$. Analogously to the proof of  item (1),
 $d_J(P)\geq q-1$ for any vertex $P$ in $J$. Applying the equation $d_{\overline{J}}(P)+d_J(P)=q(q-1)$, it follows that $d_{\overline{J}}(P)\leq (q-1)^2$. Thus
 $M_{q-1}(n)>q+1$. The proof of item (2) is complete.
\end{proof}

\begin{analysis} \emph{
Like as Theorem D, Theorem E(2) can improve some bounds from Section \ref{SecaoExatas}. We illustrate this phenomenon
by taking  $n=(q-1)^2+1$, where $q$ is a power of $2$, with $q\geq 4$.   A combination of Proposition \ref{lower} with Theorem \ref{Exatas}(5) yields
$M_{q-1}(n)\geq q+1.$
But the exact value $M_{q-1}(n)= q+2$ holds by Theorem E(2). }

\emph{It is also worth mentioning that if  $q$ is a prime power and
$q\geq 4$,  the parameters $(s,n)=(q-1,(q-1)^2+1)$ do not satisfy
the hypotheses of Proposition \ref{cond_aritmetica}.  Thus, for
these parameters,  the exact values on $M_s(n)$  supplied by
Theorem E(2)  can not be obtained from Proposition
\ref{cond_aritmetica} or Theorem A.}
\end{analysis}

A closer look reveals that a more general lower bound can be
obtained from the graphs $G(q,i,k)$, more specifically:

\begin{proposition}\label{pot_primo_inf}
Given a prime power $q$, let $i$ and $k$ be integers such that $0\leq i\leq q-1$ and $0\leq k\leq q-2$. Except when $i=0$ and $k>1$,
$$M_{q-k}((q-k-1)(q-i)+k+1)> q-i+1.$$
\end{proposition}

Propositions \ref{pot_primo_inf} and \ref{LimSup} provide us new
exact values. For instance, $M_3(6)=4$ is derived from the
parameters $(q,i,k)=(4,2,1)$, while $M_4(5)=3$ follows when
$(q,i,k)=(5,4,1)$.


\section{A table} \label{tabela}

Our results concerning the evaluation of the numbers $M_s(n)$ can be summarized in the following table.

\begin{table}[h]
\caption{Values of $M_s(n)$, $2\leq s\leq 5$ and $2\leq n\leq 17$*}
\centering
\begin{tabular}{c | p{1.1cm} p{1.1cm} p{1.4cm} p{1.1cm} p{1.1cm} p{1.4cm} p{1.4cm} p{1.4cm}} \hline
\diagbox{$s$}{$n$} & 2 & 3 & 4 & 5 & 6 & 7 & 8 & 9 \\ \hline
2 & 3$^{A,D}$ & 4$^{A,D}$ & 4-5$^{F,J,H}$ & 5$^{A,E}$ & 5-6$^{F,J}$ & 6$^{C}$ & 6-7$^{F,J}$ & 7-8$^H$ \\
3 & 3$^A$ & 3$^{A,D}$ & 3-4$^{F,J,H}$ & 4$^{A,D}$ & 4$^{A,J}$ & 5$^D$ & 5$^{A,J}$ & 5-6$^{F,H,J}$  \\
4 & 2-3$^K$ & 3$^A$ & 3$^{A,D}$ & 3$^{A,J}$ & 3-4$^F$ & 4$^{A,D}$ & 4$^{A,J}$ & 4-5$^{F,H}$  \\
5 & 2$^G$ & 2-3$^F$ & 3$^A$ & 3$^{A,D}$ & 3$^A$ & 3-4$^{F,J,H}$ & 3-4$^K$ & 4$^{A,D}$ \\
 \hline
\end{tabular}
\label{table}
\end{table}

\begin{table}[ht]
\centering
\begin{tabular}{c |p{1.0cm} p{1.0cm} p{1.1cm} p{1.1cm} p{1.0cm} p{1.5cm} p{1.6cm} p{1.5cm}} \hline
\diagbox{$s$}{$n$} & 10 & 11 & 12 & 13 & 14 & 15 & 16 & 17 \\ \hline
2 & 7-8$^H$ & 8-9$^H$ & 8-9$^J$ & 9-10$^H$ & 10$^C$ & 10-11$^H$ & 11-12$^{I,H}$ & 11-12$^{I,H}$ \\
3 & 6$^E$ & 6$^{A,B,J}$ & 6-7$^F$ & 6-7$^{H,J}$ & 7-8$^H$ & 7-8$^{F,J,H}$ & 7-8$^{I,J,H}$ & 8-9$^{I,J,H}$ \\
4 & 5$^D$ & 5$^{A,J}$ & 5$^A$ & 6$^D$ & 6$^E$ & 6$^A$ & 6-7$^{F,I,J,H}$ & 7$^E$ \\
5 & 4$^A$ & 4$^{A,J,H}$ & 4-5$^{F,J}$ & 5$^D$ & 5$^A$ & 5$^{A,J}$ & 5-6$^{F,I,J,H}$ & 6$^{D,E}$ \\
 \hline
\multicolumn{9}{l}{{\footnotesize{*Keys: \textit{A}: for Theorem A; \textit{B}: for Theorem B; \textit{C}: for Theorem C; \textit{D}: for Theorem D;}}} \\
\multicolumn{9}{l}{{\footnotesize{ \textit{E}: for Theorem E; \textit{F}: for Proposition \ref{n+1}; \textit{G}: for Proposition \ref{cond_aritmetica}; }}} \\
\multicolumn{9}{l}{{\footnotesize{ \textit{H}: for Theorem \ref{Exatas} and Propositions \ref{lower} and \ref{LimSup}; \textit{I}: for Proposition \ref{ramsey3};}}} \\
\multicolumn{9}{l}{{\footnotesize{ \textit{J}: for Propositions \ref{pot_primo_inf} and \ref{LimSup}; \textit{K}: $2\leq M_s(n)\leq M_s(n+1)$.}}} \\
\end{tabular}
\label{table2}
\end{table}

\newpage
{\bf Acknowledgments:} The first author is partially supported by Capes.
The second author is partially supported by CNPq/MCT grants: 311703/2016-0. \\



\begin{thebibliography}{99}

\bibitem{Bollobas} {B. Bollob\'{a}s, Modern graph theory. Graduate Texts in
Mathematics, 184 (Springer-Verlag, New York, 1998)}.

\bibitem{B}{S. Burr, P. Erd\"{o}s, R.J. Faudree, C.C. Rousseau, R.H. Schelp, Some complete bipartite graph-tree Ramsey numbers,
Ann. Discrete Math., 41 (1989) 79--90.}

\bibitem{Brown}{W.G. Brown, On graphs that do not contain a Thomsen graph, Canada Math. Bull., 9 (1966) 281--289.}

\bibitem{Stipp}{A.P. Burguer, P.J.P. Grobler, E.H. Stipp, J.H. van Vuuren, Diagonal Ramsey numbers in multipartite graphs,
Util. Math., 66 (2004) 137--163.}

\bibitem{Burger}{A.P. Burger, J.H. van Vuuren, Ramsey numbers in complete balanced multipartite graphs. Part I: Set numbers,
Discrete Math., 283 (2004) 37--43.}

\bibitem{Chen}{G. Chen, A result on $C_4$-star Ramsey numbers, Discrete Math., 163 (1997) 243--246.}

\bibitem{Erdos}{P. Erd\"{o}s, A. R\'{e}nyi, V.T. S\'{o}s, On a problem of graph theory, Studia Sci. Math. Hungar., 1 (1966) 215--235.}


\bibitem{HP}{D.R. Hughes, F.C. Piper, Projective Planes. Graduate Texts in Mathematics 6 (Springer-Verlag, New York, 1973)}.

\bibitem{m}{E.L. Monte Carmelo, Configurations in projective planes and quadrilateral-star Ramsey numbers, Discrete Math., 308 (2008) 3986-–3991.}

\bibitem{Parsons}{T.D. Parsons, Ramsey graphs and block designs I, Trans. Amer. Math. Soc., 209 (1975) 33--44.}

\bibitem{Pa0}{T.D. Parsons, Graphs from projective planes, Aequ. Math., 14 (1976) 167--189.}

\bibitem{Pablo}{P.H. Perondi, E.L. Monte Carmelo, Ramsey numbers in multipartite graphs arising from combinatorial designs, submitted for publication, 2017.}

\bibitem{Wu}{Y. Wu, Y. Sun, R. Zhang, S.P. Radziszowski, Ramsey numbers of $C_4$ versus wheels and stars, Graphs Combin., 31 (2015) 2437--2446.}

\bibitem{Z1}{X.M. Zhang, Y.J. Chen, T.C. Edwin Cheng, Polarity graphs and Ramsey numbers for C4 versus stars, Discrete Math., 340 (2017) 655--660.}

\bibitem{Z2}{X.M. Zhang, Y.J. Chen, T.C. Edwin Cheng, Some values of Ramsey numbers for C4 versus stars, Finite Fields Appl., 45 (2017) 73-–85.}

\end{thebibliography}
\end{document}